\documentclass[dvipdfmx]{article}
\usepackage{amsmath, amssymb, amsthm, delarray}

\usepackage[top =20mm, left=35mm,right=37mm]{geometry}
\usepackage{graphicx}
\usepackage{tikz}
\usetikzlibrary{intersections,calc,arrows.meta}

\newtheorem{thm}{Theorem}[section]
\newtheorem*{theorem*}{Theorem}

\newtheorem{prop}[thm]{Proposition}
\newtheorem{cor}[thm]{Corollary}
\theoremstyle{definition}

\newtheorem{rem}[thm]{Remark}

\numberwithin{equation}{section}
\theoremstyle{remark}

\begin{document}

\title{\textbf{Bounded intervals containing a given number of primes}}

\author{Keiju Sono}

\date{}
\allowdisplaybreaks

\maketitle 
\noindent
\begin{abstract}
Let $m$ be a non-negative integer and $f(x)$  a positive, non-decreasing function satisfying certain conditions.  We give an explicit lower bound for the number of integers  $n\leq x$ such that $\#([n, n+f(n)] \cap \mathbb{P})=m$ for  sufficiently large $x$, where $\mathbb{P}$ denotes the set of prime numbers. This work extends the results of Mastrostefano in  \cite{Mas1}, \cite{Mas2} and of Freiberg in \cite{F},  and also makes the lower bound in \cite{Mas2} explicit. In addition, we show that if the interval length is sufficiently large, then there exist infinitely many bounded intervals of the same length that contain exactly a prescribed number of primes.

\footnote[0]{2020 {\it Mathematics Subject Classification}. 11N36}
\footnote[0]{{\it Key Words and Phrases}. Small gaps between primes, Maynard-Tao sieve}
\end{abstract}

%%%%%%%%%%%%%%%%%%%%%%%%%%%%%%%%%%%%%%%%%%%%%%%%%%%%%%%%%%%%%%%%%%%%%%%%%%%%%%%%%%%%%%%%%%%%%%%%%%%%%%%%%%%%
%%%%%%%%%%%%%%%%%%%%%%%%%%%%%%%%%%%%%%%%%%%%%%%%%%%%%%%%%%%%%%%%%%%%%%%%%%%%%%%%%%%%%%%%%%%%%%%%%%%%%%%%%%%%

\section{Introduction}
The distribution of prime numbers in short intervals is a central topic in analytic number theory. In particular, since the introduction of the GPY sieve \cite{GPY1}-\cite{GPY4} and  the breakthroughs of Zhang \cite{Z} and Maynard \cite{May1}, this field has continued to develop rapidly. The main objective of this area is to clarify how many prime numbers can be contained in intervals of a given length. For example, Zhang proved that, for $H=7.0\times 10^{7}$, there exist infinitely many positive integers $n$ such that the interval $[n, n+H]$ contains at least two primes. 

This result was soon sharpened : Maynard proved the same claim for $H=600$, and the DHJ Polymath \cite{Polymath2} further reduced  $H$ to $246$. Maynard also showed that for any $m\in \mathbb{N}$, there exists a positive integer $H_{m}\asymp m^{3}e^{4m}$ such that for infinitely many positive integers $n$, the interval $[n, n+H_{m}]$ contains at least $m$ primes.

On the other hand, studying the existence of intervals of fixed length containing exactly a given number of primes is more delicate, as it requires not only adapting the Maynard-Tao sieve but also precise estimates of the number of primes. Several results are known in the cases  where  the length of the interval increases slowly as $x \to \infty$. Freiberg \cite{F} proved that for any fixed real number $\lambda$ and any non-negative integer $m$, if $x$ is sufficiently large in terms of $\lambda$ and $m$, then 
\[
\# \{n \leq x \; | \; \# ([n, n+\lambda \log n] \cap \mathbb{P})=m \}\geq x^{1-\varepsilon (x)},
\]
where $\mathbb{P}$ denotes the set of primes and $\varepsilon (x)$ is a certain function that tends to zero as $x\to \infty$. Mastrostefano \cite{Mas1} sharpened Freiberg's result and showed that the left-hand side of the above is $\gg _{m}x/\log x$.  Shortly thereafter, he also proved that there exists a positive proportion of integers $n\leq x$ such that the interval $[n, n+\lambda \log n]$ contains exactly $m$ primes. More specifically, he proved that 
\[
d_{\lambda ,m}(x):=\frac{\# \{n\leq x \; | \; \# ([n, n+\lambda \log n] \cap \mathbb{P})=m \}}{x}\gg \lambda ^{k+1}e^{-Dk^{4}\log k}
\]
for a certain absolute constant $D>0$, if $x$ is sufficiently large in terms of $m$ and $\lambda$, where $k=C\exp (49m/C^{\prime})$ for some positive constants $C$, $C^{\prime}$. 

The main aim of this paper is to provide an explicit evaluation of the proportion of positive integers $n \leq x$ for which more general intervals of the form $[n, n+f(n)]$ contain exactly $m$ primes, where $f$ is a non-decreasing function that satisfies certain conditions. For two functions $f(x)$ and $g(x)$, we write $f(x)\ll g(x)$, or equivalently, $f(x)=O(g(x))$, if there exists a constant $C>0$, independent of $x$, such that $|f(x)|\leq Cg(x)$ holds for all sufficiently large $x$. The constant $C$ is called the {\it implied constant}.

We prove the following theorem. 
\begin{thm}
Let $m$ be a sufficiently large positive integer and put $k=\lfloor \exp (245m) \rfloor$. Let $f(x)$ be a non-decreasing function with $f(x)\geq 8e^{\gamma}k\log k$ $ ($$\gamma =0.57721\ldots$ denotes the Euler-Mascheroni constant$)$  and
\begin{equation}
\label{21.5}
f(x) \leq \frac{C}{k^{4}\log ^{2}k}\log x
\end{equation}
for any fixed absolute constant $C>0$ and any sufficiently large $x$, and 
\begin{equation}
\label{28}
f(32xf(x))\leq 2f(x).
\end{equation}
Put $g(x):=20x f(x)$. Then, for any sufficiently large $x$, we have
\begin{equation}
\label{MT}
\# \{n\leq x \; | \; \#([n, n+f(n)] \cap \mathbb{P})=m \}\gg \frac{x}{k^{2}\log k}e^{-3k^{4}\log k}\left( \frac{f(g^{-1}(x))}{\log g^{-1}(x)} \right)^{k}.
\end{equation}
Here, the implied constant is independent of $m$ and $x$. Furthermore, the same estimate remains valid even under the additional condition  that "all primes in the interval $[n, n+f(n)]$ are at least $\gg \frac{f(g^{-1}(x))}{k\log k}$ apart" to the set of the  left-hand side.
\end{thm}
As another example besides Mastrostefano's case $f(x)=\lambda \log x$, $f(x)=\log \log x$ also satisfies the conditions of Theorem 1.1. As a corollary to the above theorem, we obtain the following result, which is an explicit version of Mastrostefano's theorem in \cite{Mas2}.
\begin{cor}
Let $m$ be a sufficiently large positive integer and set $k=\lfloor \exp (245m ) \rfloor$. Then, for any constant $\lambda$ with $0<\lambda \ll k^{-4}\log ^{-2}k$ and any sufficiently large $x$, we have
\[
\# \{n\leq x \; | \; \#([n, n+\lambda \log n ] \cap \mathbb{P})=m \}\gg x \frac{\lambda ^{k}}{k^{2}\log k}e^{-3k^{4}\log k}.
\]
The implied constant is independent of $m$ and $x$. 
\end{cor}
The function $g^{-1}(x)$ is somewhat difficult to handle. The following estimate can be used as a simplified version. 
\begin{cor}
Let $m$ be a sufficiently large positive integer and put $k=\lfloor \exp (245m ) \rfloor$. Suppose that $f$ is a differentiable function that satisfies the conditions of Theorem 1.1 and $\frac{f^{\prime}(x)}{f(x)}\leq \frac{1}{x\log x}$. Then, for any sufficiently large $x$, we have 
\[
\# \{n \leq x \; | \; \#([n, n+f(n)] \cap \mathbb{P})=m \}\gg \frac{xe^{-3k^{4}\log k}}{k^{2}\log k}\left( \frac{f(x)}{\log x} \right)^{k}.
\]
The implied constant is independent of $x$ and $m$. In particular, if $f$ satisfies these conditions, then there exist infinitely many positive integers $n$ that satisfy $\# ([n, n+f(n)] \cap \mathbb{P})=m$. 
\end{cor}

Finally, in case where the interval lengths are bounded, the following result is obtained. 
\begin{cor}
Let $m$ be a sufficiently large positive integer and put $k=\lfloor \exp (245m ) \rfloor$. Let $C_{k}$ be an arbitrary positive constant such that $C_{k}\geq 8e^{\gamma}k\log k$. Then, for any sufficiently large $x$, we have
\[
\# \{n \leq x \; | \; \#([n, n+C_{k}] \cap \mathbb{P})=m \}\gg \frac{xe^{-3k^{4}\log k}}{k^{2}\log k}\left( \frac{C_{k}}{\log \frac{x}{20C_{k}}} \right)^{k}.
\]
The implied constant is independent of $x$ and $m$. In particular, there exist infinitely many positive integers $n$ that satisfies $\# ([n, n+C_{k}] \cap \mathbb{P})=m$. 
\end{cor}
\begin{rem}
Maynard \cite{May1} showed that there exist infinitely many positive integers $n$ such that the bounded interval $[n, n+C]$ contains at least $m$ primes if $C \geq m^{3}e^{4m}$. While the value $m^{3}e^{4m}$ is significantly smaller than $8e^{\gamma}k\log k  (\asymp me^{245m})$, it seems difficult to prove the existence of infinitely many intervals of fixed length $\asymp m^{3}e^{4m}$ containing **exactly** $m$ primes, based solely on Maynard's argument.
\end{rem}

%%%%%%%%%%%%%%%%%%%%%%%%%%%%%%%%%%%%%%%%%%%%%%%%%%%%%%%%%%%%%%%%%%%%%%%%%%%%%%%%%%%%%%%%%%%%%%%%%%%%%%%%%%%%
%%%%%%%%%%%%%%%%%%%%%%%%%%%%%%%%%%%%%%%%%%%%%%%%%%%%%%%%%%%%%%%%%%%%%%%%%%%%%%%%%%%%%%%%%%%%%%%%%%%%%%%%%%%%

\section{Application of Maynard's sieve for linear forms}
Let $\mathbb{P}$ denote the set of all primes. For any set $S\subset \mathbb{N}$, denote its characteristic function by $\mathbf{1}_{S}$.  For $k$ distinct integers $h_{i}$ ($i=1, \ldots ,k$), we say that the set ${\cal H}=\{h_{1}, \ldots ,h_{k} \}$ is admissible if the set of residue classes $\{h_{i}\; (\textrm{mod}\; p) \; | \; i=1, \ldots ,k \}$ is a proper subset of $\mathbb{Z}/p\mathbb{Z}$ for any $p\in \mathbb{P}$. More generally, the set ${\cal L}=\{L_{1}, \ldots ,L_{k}\}$ consisting of $k$ distinct linear forms
\[
L_{i}(n)=a_{i}n+b_{i} \quad (a_{i}, b_{i}\in \mathbb{Z})
\]
$(i=1, \ldots ,k)$ is called admissible if for any $p\in \mathbb{P}$, there exists some integer $n=n_{p}\in \mathbb{Z}$ such that
\[
L_{1}(n)\cdots L_{k}(n)\not \equiv 0 \quad (\mathrm{mod}\; p).
\]
In particular, if these linear forms are of the form $L_{i}(n)=gn+h_{i}$ ($g\in \mathbb{N}$, $h_{1}, \ldots ,h_{k}\in \mathbb{Z}$) for $i=1, \ldots ,k$, then it follows that ${\cal L}=\{L_{1}, \ldots ,L_{k}\}$ is admissible if and only if ${\cal H}=\{h_{1}, \ldots ,h_{k} \}$ is so. Let $\varphi$ denote Euler's totient function. The following proposition is an adaptation of Maynard's sieve (\cite{May2}, Proposition 6.1) to our situation.

\begin{prop}
Let $x$ be a sufficiently large real number and $k$ a sufficiently large positive integer with $k\leq (\log x)^{\frac{1}{5}}$. Suppose that ${\cal H}=\{h_{1}, \ldots ,h_{k}\}$ is an admissible set with $0<h_{1}<\cdots <h_{k}<f(x)$, where $f(x)$ is a non-decreasing function with $f(x)\geq 1$, $f(x)\ll _{\varepsilon} x^{\varepsilon}$ as $x\to \infty$, and put ${\cal L}=\{gn+h_{1}, \ldots ,gn+h_{k}\}$. Put $\eta =\frac{c}{500k^{2}}$, where $c>0$ is a fixed constant. Put $R:=x^{\frac{1}{24}}$, and suppose $k^{-4} \leq \rho \leq \frac{1}{80}$. Let $B$ be an integer which is either $1$ or a prime number satisfying $\log \log x^{\eta}\ll B \ll x^{2\eta}$. Define a subset of positive integers $S(\rho ,B)$ by 
\[
S(\rho ,B)=\{n \in \mathbb{N}\; | \; p|n \Rightarrow (p>x^{\rho} \; \mathrm{or} \; p|B) \}.
\]
Finally, let $g$ be a square-free positive integer in $[6f(x), 8f(x)]$ which is coprime to $B$. Then, there exists a non-negative weight $w_{n}=w_{n}({\cal H})$ that satisfies 
\begin{equation}
\label{2}
w_{n}\ll (\log R)^{2k}\prod _{i=1}^{k}\underset{p\nmid B}{\prod _{p|gn +h_{i}}}4
\end{equation}
and the following estimates hold. 
\begin{equation}
\label{3}
\sum _{x<n \leq 2x}w_{n}=\left(1+O\left( \frac{1}{(\log x)^{\frac{1}{10}}} \right) \right)\frac{B^{k}}{\varphi (B)^{k}}\mathfrak{S}_{B}({\cal H})x(\log R)^{k}I_{k},
\end{equation}
\begin{equation}
\label{4}
\begin{aligned}
\sum _{x<n \leq 2x}\mathbf{1}_{\mathbb{P}}(gn+h_{i})w_{n}&\geq \left(1+O\left( \frac{1}{(\log x)^{\frac{1}{10}}} \right) \right)\frac{B^{k-1}}{\varphi (B)^{k-1}}\mathfrak{S}_{B}({\cal H})  \frac{\varphi (g)}{g}(\log R)^{k+1}J_{k}\sum _{x<n \leq 2x}\mathbf{1}_{\mathbb{P}}(gn+h_{i}) \\
& \quad +O\left(\frac{B^{k}}{\varphi (B)^{k}}\mathfrak{S}_{B}({\cal H})x(\log R)^{k-1}I_{k} \right), 
\end{aligned}
\end{equation}
$(i=1, \ldots ,k)$, and 
\begin{equation}
\label{5}
\begin{aligned}
\sum _{x<n \leq 2x}\mathbf{1}_{S(\rho ,B)}(gn+h)w_{n}& \ll \rho ^{-1}\frac{\Delta _{\cal L}}{\varphi (\Delta _{\cal L})}\frac{B^{k+1}}{\varphi (B)^{k+1}}\mathfrak{S}_{B}({\cal H})x(\log R)^{k-1}I_{k}
\end{aligned}
\end{equation}
for any integer $h\in [1, 5f(x)]\backslash {\cal H}$, where
\[
\Delta _{\cal L}:=g^{k+1}\prod _{i=1}^{k}|h-h_{i}| \neq 0.
\]
Finally, 
\begin{equation}
\label{6}
\sum _{x<n \leq 2x} \left(\sum _{\substack{p|gn+h _{i} \\ p<x^{\rho}, p\nmid B}} 1 \right)w_{n}\ll \rho ^{2}k^{4}(\log k)^{2}\frac{B^{k}}{\varphi (B)^{k}}\mathfrak{S}_{B}({\cal H})x(\log R)^{k}I_{k}
\end{equation}
for $i=1, \ldots ,k$. Here, $I_{k}$ and $J_{k}$ are some quantities depending only on $k$, $\mathfrak{S}_{B}({\cal H})=\mathfrak{S}_{g,B}({\cal H})$ is a quantity that depends only on $g$, $B$ and ${\cal H}$ satisfying 
\begin{equation}
\label{7}
\mathfrak{S}_{B}({\cal H})\gg \frac{1}{\exp (O(k))},
\end{equation}
and
\begin{equation}
\label{9}
I_{k}\gg (2k\log k)^{-k}, \quad J_{k}\gg \frac{\log k}{k}I_{k}.
\end{equation}
\end{prop}
(For detailed definitions of the notation, see Proposition 6.1 of \cite{May2}.) 
\begin{rem}
The implied constant in (\ref{9}) can be determined specifically. At the last of this paper,  we will see that one can choose $w_{n}$ so that 
\begin{equation}
\label{IJ}
J_{k}\geq \frac{1}{5}\frac{\log k}{k}I_{k}
\end{equation}
holds for any sufficiently large $k$. 
\end{rem}
Following Mastrostefano \cite{Mas2}, consider the sum 
\begin{equation}
\label{10}
S=\sum _{\cal H}^{\quad \quad *}\sum _{x<n \leq 2x}S({\cal H}, n),
\end{equation}
where
\begin{equation}
\label{11}
\begin{aligned}
S({\cal H}, n)&:=\left( \sum _{i=1}^{k}\mathbf{1}_{\mathbb{P}}(gn+h_{i})-m-k \sum _{i=1}^{k}\sum _{\substack{p|gn+h_{i} \\ p\leq x^{\rho}, p\nmid B}}1-k\sum _{\substack{h\leq 5f(x) \\ (h,g)=1, h \not \in {\cal H}}}\mathbf{1}_{S(\rho ,B)}(gn+h) \right)w_{n}({\cal H}),
\end{aligned}
\end{equation}
for $g$ in the Proposition 2.1 and $\sum _{\cal H}^{*}$ denotes the sum over all admissible sets ${\cal H}=\{h_{1}, \ldots ,h_{k}\}$ with $k$ elements satisfying 
\[
0<h_{1}<h_{2}<\cdots <h_{k}<f(x),
\]
\[
|h_{i}-h_{j}| > \frac{f(x)}{C_{0}}
\]
for any $1\leq i \neq j \leq k$, where $k=O(1)$ and $C_{0}=C_{0}(k)$ is a quantity which depends only on $k$. (We will determine this specifically later.) The weight $w_{n}({\cal H})$ is that of the above proposition. Since the terms with $S({\cal H}, n)>0$ only give positive contributions to $S$, and for such $n$ and ${\cal H}$, 
\begin{equation}
\label{12}
k\sum _{i=1}^{k}\sum _{\substack{p|gn+h_{i} \\ p\leq x^{\rho}, p\nmid B}}1=0
\end{equation}
must hold. Hence in (\ref{2}), the conditions $p | gn+h_{i}$ and $p\nmid B$ imply $p>x^{\rho}$. Therefore, by Proposition 2.1, it follows that 
\[
w_{n}\ll (\log R)^{2k}\prod _{i=1}^{k}\prod _{\substack{p|gn+h_{i} \\ p>x^{\rho}, p\nmid B}}4.
\]
Since $g\in [6f(x), 8f(x)]$ and $0<h_{i}<f(x)$, we have $gn+h_{i}\leq 17xf(x)$. Put
\[
n_{i}:=\# \{p\in \mathbb{P} \; | \; p | gn+h_{i}, \; p>x^{\rho} \}.
\]
Then 
\[
(x^{\rho})^{n_{i}}\leq gn+h_{i}\leq 17xf(x).
\]
Combining this with $f(x)\ll x^{\varepsilon}$, it follows that $n_{i}\leq \frac{2}{\rho}$ for $i=1, \ldots ,k$. Hence
\begin{align*}
w_{n}&\ll (\log R)^{2k}\prod _{i=1}^{k}4^{\frac{2}{\rho}} \ll (\log x)^{2k}\exp \left( \frac{4k}{\rho} \log 2 \right).
\end{align*}
Therefore, 
\begin{equation}
\label{13}
S\ll k (\log x)^{2k}\exp \left( \frac{4k\log 2}{\rho} \right)\sum _{{\cal H}}^{\quad \quad *}\sum _{x<n \leq 2x}\mathbf{1}_{S({\cal H}, n)>0}.
\end{equation}
Furthermore, in order that $S({\cal H},n)>0$ holds,   in addition to the condition (\ref{12}),  the following must also be satisfied. 
\begin{equation}
\label{14}
\sum _{i=1}^{k}\mathbf{1}_{\mathbb{P}}(gn+h_{i})-m>0,
\end{equation}
and
\begin{equation}
\label{15}
k \sum _{\substack{h \leq 5f(x) \\ (h,g)=1, h \not \in {\cal H}}}\mathbf{1}_{S(\rho ,B)}(gn+h)=0.
\end{equation}

Note that condition (\ref{12}) implies that all of $gn+h_{1}, \ldots ,gn+h_{k}$ are divisible only by a prime $p$ with $p>x^{\rho}$ or $p|B$. 

Condition (\ref{14}) ensures that at least $m+1$ of $gn+h_{1}, \ldots ,gn+h_{k}$ are primes. 

Finally, condition (\ref{15}) implies that for any integers $1\leq h \leq 5f(x)$ with $h\not \in {\cal H}$, $gn+h$ is divisible by some prime $p\leq x^{\rho}$. (In fact, this statement directly follows from the definition if $(h,g)=1$. If $(h,g)>1$, then $gn+h$ is divisible by a common prime factor of $g$ and $h$, since $g,h \ll f(x)\ll x^{\varepsilon}<x^{\rho}$.) In particular, all $gn+h$ with $h\in [1, 5f(x)]$, $h\not \in {\cal H}$ are not primes. 

Therefore, all primes in $[gn, gn+5f(x)]$ are contained in the set $\{gn+h_{1}, \ldots ,gn+h_{k}\}$ for some admissible set ${\cal H}=\{h_{1}, \cdots ,h_{k}\}$, and for any fixed $n$, there are no distinct admissible sets that give positive contributions to $S$.  Consequently,
\begin{equation}
\label{16}
S\ll k (\log x)^{2k}\exp \left( \frac{4k\log 2}{\rho} \right) |I(x)|.
\end{equation}
Here, the set  $I(x)$ consists of the intervals  of the form $[gn, gn+5f(x)]$, $x<n\leq 2x$ which have the property that there exists a unique admissible set ${\cal H}=\{h_{1}, \ldots ,h_{k}\}$ with $0<h_{1}<\ldots <h_{k}<f(x)$, $|h_{i}-h_{j}|> \frac{f(x)}{C_{0}}$ ($i\neq j$) such that 
\[
|[gn, gn+5f(x)]\cap \mathbb{P}|=|\{gn+h_{1}, \ldots ,gn+h_{k} \}\cap \mathbb{P}| \geq m+1
\]
holds. Moreover, these intervals in $I(x)$ are disjoint, since such intervals are expressed by  $[gn, gn+5f(x)]$, $[gm, gm+5f(x)]$ ($n<m$) and 
\[
gn+5f(x)<g(n+1)\leq gm,
\]
since $g\geq 6f(x)$. 

Next, we give a lower bound for $S$. By substituting (\ref{3})-(\ref{6}) into (\ref{11}) and combining (\ref{10}), we have 
\begin{equation}
\label{17}
\begin{aligned}
S&\geq \left(1+O\left( \frac{1}{(\log x)^{\frac{1}{10}}} \right) \right)\sum _{{\cal H}}^{\quad \quad *}\Bigg[\frac{B^{k-1}}{\varphi (B)^{k-1}}\mathfrak{S}_{B}({\cal H})\frac{\varphi (g)}{g}(\log R)^{k+1}J_{k}\sum _{x<n\leq 2x}\sum _{i=1}^{k}\mathbf{1}_{\mathbb{P}}(gn+h_{i})  \\
& \quad \quad \quad \quad +O\left( \frac{kB^{k}}{\varphi (B)^{k}}\mathfrak{S}_{B}({\cal H})x(\log R)^{k-1}I_{k}\right)  -m\frac{B^{k}}{\varphi (B)^{k}}\mathfrak{S}_{B}({\cal H})x(\log R)^{k}I_{k} \\
& \quad \quad \quad \quad +O\left( \rho ^{2}k^{6}(\log k)^{2}\frac{B^{k}}{\varphi (B)^{k}}\mathfrak{S}_{B}({\cal H})x(\log R)^{k}I_{k} \right) \\
& \quad \quad \quad \quad  +O\left( k\rho ^{-1}\frac{B^{k+1}}{\varphi (B)^{k+1}} \mathfrak{S}_{B}({\cal H})x(\log R)^{k-1}I_{k} \sum _{\substack{h\leq 5f(x) \\ (h,g)=1, h\not \in {\cal H}}} \frac{\Delta _{{\cal L}}}{\varphi (\Delta _{ {\cal L}})}   \right) \Bigg].
\end{aligned}
\end{equation}
By (2.7) of \cite{Mas1}, we have
\begin{equation}
\label{18}
\frac{\varphi (B)}{B}\frac{\varphi (g)}{g}\sum _{i=1}^{k}\sum _{x<n \leq 2x}\mathbf{1}_{\mathbb{P}}(gn+h_{i})>\frac{kx}{2\log x},
\end{equation}
and by replacing $\lambda \log x$ with $f(x)$ in the proof of Lemma 3.1 of the same paper, by almost the same argument it follows that
\begin{equation}
\label{19}
 \sum _{\substack{h\leq 5f(x) \\ (h,g)=1, h\not \in {\cal H}}} \frac{\Delta _{{\cal L}}}{\varphi (\Delta _{ {\cal L}})}  \ll (\log k)f(x)
\end{equation}
for $g\in [6f(x), 8f(x)]$. By substituting (\ref{18}), (\ref{19}) into (\ref{17}), we obtain 
\begin{align*}
S&\gg \sum _{{\cal H}}^{\quad \quad *}\frac{B^{k}}{\varphi (B)^{k}}\mathfrak{S}_{B}({\cal H})x(\log R)^{k} \\
& \quad \times \left[ kJ_{k}\frac{\log R}{2\log x}+O\left( \frac{kI_{k}}{\log R} \right)-mI_{k}  +O(\rho ^{2}k^{6}(\log k)^{2}I_{k})+O\left( \frac{k\log k}{\rho}\frac{B}{\varphi (B)}\frac{f(x)}{\log R}I_{k} \right) \right].
\end{align*}
Set
\begin{equation}
\label{20}
\rho =k^{-3}(\log k)^{-1}.
\end{equation}
Substituting this and $R=x^{\frac{1}{24}}$, and using (\ref{IJ}), we obtain 
\begin{equation}
\label{21}
\begin{aligned}
S&\gg \sum _{{\cal H}}^{\quad \quad *}\frac{B^{k}}{\varphi (B)^{k}}\mathfrak{S}_{B}({\cal H})x(\log R)^{k}  \left[ \frac{I_{k}\log k}{240} +O\left( \frac{kI_{k}}{\log x} \right)-mI_{k}  +O(I_{k})+O\left( k^{4}\log ^{2}k\frac{f(x)}{\log x} I_{k} \right) \right].
\end{aligned}
\end{equation}
If $f(x)$ satisfies 
\begin{equation}
\label{21.5}
f(x)\leq \frac{C}{k^{4}\log ^{2}k}\log x
\end{equation}
for an arbitrarily fixed absolute constant $C$ and any $x\geq  x_{0}=x_{0}(k)$, and
\begin{equation}
\label{22}
k=\lfloor \exp (245m ) \rfloor ,
\end{equation}
then for any sufficiently large $m$, the sum of the terms in brackets in (\ref{21}) becomes $\gg I_{k}$ uniformly for $x>10$, and combining this with $\frac{B}{\varphi (B)}\asymp 1$, it follows that
\begin{equation}
\label{23}
S\gg x(\log R)^{k}I_{k}\sum _{{\cal H}}^{\quad \quad *}\mathfrak{S}_{B}({\cal H}).
\end{equation}
Since
\[
I_{k}\gg (2k\log k)^{-k}, \quad \mathfrak{S}_{B}({\cal H})\gg \exp(-O(k))
\]
by (\ref{7}) and (\ref{9}), we obtain 
\begin{equation}
\label{24}
\begin{aligned}
S\gg x(\log x)^{k}e^{-k^{2}}\sum _{\substack{{\cal H}=\{h_{1}, \ldots ,h_{k} \}:\mathrm{admissible} \\ 0<h_{1}<\cdots <h_{k}<f(x) \\ |h_{i}-h_{j}|> \frac{f(x)}{C_{0}} \; (1\leq i\neq j \leq k)}}1.
\end{aligned}
\end{equation}
We evaluate the number of admissible sets in (\ref{24}) by using the method in \cite{Mas2}. First, from the interval $[0, f(x)]$, for each prime number $p$ less than or equal to $k$,  in ascending order, remove all integers with the smallest number of elements in the residue class modulo $p$, and denote the resulted set by ${\cal A}$. By Mertens formula, we have
\[
|{\cal A}|\geq f(x)\prod _{p\leq k}\left(1-\frac{1}{p} \right)\geq  \frac{e^{-\gamma}f(x)}{2\log k}.
\]
The set of $k$ distinct integers $\{h_{1}, \ldots ,h_{k}\}$ arbitrarily chosen from ${\cal A}$ is admissible. Next, we construct a subset $\{h_{1}, \ldots ,h_{k}\}$ that satisfies $|h_{i}-h_{j}|> \frac{f(x)}{C_{0}}$ $(\forall i \neq j$) from ${\cal A}$. Take $h_{1}\in {\cal A}$ arbitrarily. There are $|{\cal A}|$ ways to choose such $h_{1}$. Put ${\cal A}_{1}:={\cal A}$. Next, put
\[
{\cal A}_{2}:={\cal A}_{1}\backslash \left[h_{1}-\left\lfloor \frac{f(x)}{C_{0}} \right\rfloor ,h_{1}+\left\lfloor \frac{f(x)}{C_{0}} \right\rfloor \right]
\]
and choose $h_{2}\in {\cal A}_{2}$ arbitrarily. There are
\[
|{\cal A}_{2}|\geq  |{\cal A}|-2\left\lfloor \frac{f(x)}{C_{0}} \right\rfloor
\]
ways to choose such $h_{2}$. Inductively, for $1\leq i \leq k$, once ${\cal A}_{i-1}$ are defined and $h_{i-1}\in {\cal A}_{i-1}$ are chosen,  define  ${\cal A}_{i}$ by
\[
{\cal A}_{i}:={\cal A}_{i-1}\backslash \left[h_{i-1}-\left\lfloor \frac{f(x)}{C_{0}} \right\rfloor ,h_{i-1}+\left\lfloor \frac{f(x)}{C_{0}} \right\rfloor \right]
\]
and choose $h _{i}\in {\cal A}_{i}$ arbitrarily. There are 
\[
|{\cal A}_{i}| \geq |{\cal A}|-2(i-1)\left\lfloor \frac{f(x)}{C_{0}} \right\rfloor
\]
ways to choose such $h_{i}$. In this way, an admissible set ${\cal H}=\{h_{1}, \ldots ,h_{k}\}$ is constructed, and this satisfies the condition $|h_{i}-h_{j}|> \frac{f(x)}{C_{0}}$ for $i\neq j$, and taking the order $0<h_{1}<h_{2}<\ldots <h_{k}$ into account,  the number of such admissible sets is 
\begin{align*}
\frac{1}{k!}\prod _{i=1}^{k}|{\cal A}_{i}|&\geq \frac{1}{k^{k}}\prod _{i=1}^{k}\left(|{\cal A}|-2(i-1)\left\lfloor \frac{f(x)}{C_{0}} \right\rfloor \right) \\
&\geq \frac{1}{k^{k}}\left(\frac{e^{-\gamma}f(x)}{2\log k}-2(k-1)\left\lfloor \frac{f(x)}{C_{0}} \right\rfloor \right)^{k}.
\end{align*}
We set
\begin{equation}
\label{C_{0}}
C_{0}=8e^{\gamma}k\log k.
\end{equation}
Then
\[
\frac{1}{k!}\prod _{i=1}^{k}|{\cal A}_{i}| \geq \frac{1}{k^{k}}\left( \frac{e^{-\gamma}f(x)}{4\log k} \right)^{k} \geq e^{-k^{2}}f(x)^{k}.
\]
This gives a lower bound for the sum in (\ref{24}). Therefore,
\begin{equation}
\label{25}
S\gg x(\log x)^{k}e^{-2k^{2}}f(x)^{k}.
\end{equation}
Combining this with (\ref{20}) and  (\ref{16}), we have 
\[
k(\log x)^{2k}\exp \left(4(\log 2)k^{4}\log k \right)|I(x)| \gg x(\log x)^{k}e^{-2k^{2}}f(x)^{k}.
\]
Consequently, we have
\begin{equation}
\label{26}
|I(x)|\gg \frac{x}{k}e^{-3k^{4}\log k}\left( \frac{f(x)}{\log x} \right)^{k}.
\end{equation}
Here, the implied constant above is independent of $k, x$.

%%%%%%%%%%%%%%%%%%%%%%%%%%%%%%%%%%%%%%%%%%%%%%%%%%%%%%%%%%%%%%%%%%%%%%%%%%%%%%%%%%%%%%%%%%%%%%%%%%%%%%%%%%%%
%%%%%%%%%%%%%%%%%%%%%%%%%%%%%%%%%%%%%%%%%%%%%%%%%%%%%%%%%%%%%%%%%%%%%%%%%%%%%%%%%%%%%%%%%%%%%%%%%%%%%%%%%%%%

\section{Evaluation of the number of intervals containing a given number of primes}
Let $I$ be an interval which is an element of $I(x)$ in (\ref{16}) and $m, k$ the integers  introduced in the previous section. Then, for some integer $n$ with $x<n\leq 2x$ and a unique admissible set ${\cal H}=\{h_{1}, \ldots ,h_{k} \}$ with $0<h_{1}<\ldots <h_{k}<f(x)$ and $|h_{i}-h_{j}|>\frac{f(x)}{C_{0}}$
($\forall i \neq j$)  ($C_{0}$ is given by (\ref{C_{0}})), we have $I=[gn, gn+5f(x)]$ and 
 \[
 \#([gn, gn+5f(x)] \cap \mathbb{P})=\# ( \{gn+h_{1}, \ldots ,gn+h_{k} \} \cap \mathbb{P})\geq m+1.
 \]
 Put $N_{j}=gn+j$, $I_{j}=[N_{j},N_{j}+f(N_{j})]$ for $j=0, \ldots , \lfloor f(N_{0}) \rfloor$. Then the followings hold. \\
{\bf Fact  1}. $I_{j}\subset I$ for $j=0, \ldots , \lfloor f(N_{0}) \rfloor$. In fact, $N_{j}=gn+j \geq gn$, and  since $f$ is non-decreasing and $f(N_{0})=f(gn)<gn$, 
\begin{align*}
N_{j}+f(N_{j})&=gn+j+f(gn+j) \\
&\leq gn+\lfloor f(N_{0}) \rfloor +f(gn+\lfloor f(N_{0}) \rfloor ) \\ 
&\leq gn+f(gn)+f(2gn) \\
&\leq gn+5f(x)
 \end{align*}
for $g \in [6f(x), 8f(x)]$. Note that the last inequality holds if 
\begin{equation}
\label{27}
f(16xf(x))+f(32xf(x))\leq 5f(x),
\end{equation}
and this inequality holds when $f$ is non-decreasing and satisfies (\ref{28}).   \\
{\bf Fact 2}. If $j=h_{1}$, then 
\begin{equation}
\label{29}
I_{j}\cap \{gn+h_{1}, \ldots ,gn+h_{k} \}=\{gn+h_{1}, \ldots ,gn+h_{k} \}.
\end{equation}
 In fact, 
 \begin{align*}
 I_{h_{1}}=[N_{h_{1}}, N_{h_{1}}+f(N_{h_{1}})]=[gn+h_{1}, gn+h_{1}+f(gn+h_{1})],
 \end{align*}
 and $h_{1}+f(gn+h_{1})\geq h_{k}$,  since the left-hand side is at least $f(6xf(x))$ whereas the right-hand side is less than $f(x)$.  \\
{\bf Fact  3}. If $j=\lfloor f(N_{0}) \rfloor$, then 
 \begin{equation}
 \label{30}
 I_{j}\cap \{gn+h_{1}, \ldots ,gn+h_{k}\} =\emptyset .
 \end{equation}
 In fact, we have
 \[
 I_{\lfloor f(N_{0}) \rfloor}=[N_{\lfloor f(N_{0}) \rfloor}, N_{\lfloor f(N_{0}) \rfloor}+f(\lfloor f(N_{0}) \rfloor ) ] =[gn+\lfloor f(gn) \rfloor , gn+\lfloor f(gn) \rfloor +f(\lfloor f(gn) \rfloor )],
 \]
 and since $f(x)$ is non-decreasing, $f(x)\geq 1$ and $g\geq 6f(x)$, it follows that 
 \[
 \lfloor f(gn) \rfloor \geq \lfloor f(6xf(x)) \rfloor >f(x)>h_{k}.
 \]
 Hence (\ref{30}) holds. \\
{\bf Fact 4}. If 
 \begin{equation}
 \label{31}
 |I_{j}\cap \mathbb{P}|>|I_{j+1}\cap \mathbb{P}|, 
 \end{equation}
 then 
 \begin{equation}
 \label{32}
  |I_{j}\cap \mathbb{P}|=|I_{j+1}\cap \mathbb{P}|+1.
 \end{equation}
 In fact, since
 \[
 I_{j}=[gn+j, gn+j+f(gn+j)], \quad I_{j+1}=[gn+j+1, gn+j+1+f(gn+j+1)],
 \]
 the inequality (\ref{31}) holds only when $gn+j$ is a prime and the semi-open interval $(gn+j+f(gn+j), gn+j+1+f(gn+j+1)]$ does not contain any prime. In this case, $I_{j}$  contains one more prime number than $I_{j+1}$ by $gn+j$. Hence (\ref{32}) holds. \\
 Put 
 \[
 \tilde{j}:=\max \{ 0\leq j \leq \lfloor f(N_{0}) \rfloor \; \; \vline  \; \;   |I_{j}\cap \mathbb{P}| \geq m+1 \}.
 \]
 Then $N_{\tilde{j}}=gn+\tilde{j}$ is a prime and $|I_{\tilde{j}+1}\cap \mathbb{P}|=m$. Now, by this and the properties of ${\cal H}$, the following fact holds. \\
 {\bf Fact 5}. For any integers $1\leq l \leq \left\lfloor \frac{f(x)}{C_{0}} \right\rfloor$,  we have
 \begin{equation}
 \label{jtilde}
| I_{\tilde{j}+l}\cap \mathbb{P}|=m.
 \end{equation}
 In fact, since
 \[
 I_{\tilde{j}+l}=[gn+\tilde{j}+l, gn+\tilde{j}+l+f(gn+\tilde{j}+l)]
 \]
 and $gn+\tilde{j}=N_{\tilde{j}}\in \mathbb{P}$, there exists some integer $1\leq n(\tilde{j})\leq k$ for which 
 \[
 gn+\tilde{j}=gn+h_{n(\tilde{j})}
 \]
 holds. Furthermore, since
 \[
 h_{n(\tilde{j})+1}>h_{n(\tilde{j})}+\left\lfloor \frac{f(x)}{C_{0}} \right\rfloor \geq h_{n(\tilde{j})}+l,
 \]
 we have $gn+\tilde{j}+l<gn+h_{n(\tilde{j})+1}$. Hence the interval $(gn+\tilde{j}, gn+\tilde{j}+l]$ does not contain any prime. In addition, since $|I_{\tilde{j}+l}\cap \mathbb{P}|\leq m$, the interval 
 $(gn+\tilde{j}+f(gn+\tilde{j}), gn+\tilde{j}+l+f(gn+\tilde{j}+l)]$ also does not contain any prime. By the above facts, it follows that
 \[
 ((I_{\tilde{j}}\backslash I_{\tilde{j}+l}) \cup (I_{\tilde{j}+l} \backslash I_{\tilde{j}}))\cap \mathbb{P}=\{gn+\tilde{j} \}.
 \]
 Hence $I_{\tilde{j}}$ contains one more prime than $I_{\tilde{j}+l}$, that is  $gn+\tilde{j}$, and we know that $|I_{\tilde{j}}|=m+1$. So, the statement of Fact 5 holds. \\

 From this observation, it follows that for any $I\in I(x)$, there exist $\left\lfloor \frac{f(x)}{C_{0}} \right\rfloor$ distinct intervals  of the form $[N, N+f(N)]$ ($N\leq gn+\tilde{j}+\lfloor \frac{f(x)}{C_{0}} \rfloor \leq 20xf(x)$) that contain exactly $m$ primes. (Here, we used $gn\leq 16xf(x)$ and $\tilde{j}\leq f(N_{0})=f(gn)\leq f(16xf(x))\leq 2f(x)$.) Therefore, by (\ref{26}) with $C_{0}$ in  (\ref{C_{0}}), we have
 \begin{align*}
 \# \{N\leq 20x f(x) \; | \; \# ([N, N+f(N)] \cap \mathbb{P})=m \}& \gg \left\lfloor \frac{f(x)}{C_{0}} \right\rfloor |I(x)| \\
 &\gg \frac{xf(x)}{k^{2}\log k}e^{-3k^{4}\log k}\left( \frac{f(x)}{\log x} \right)^{k}.
 \end{align*}
 By putting $X=g(x)=20xf(x)$, we obtain 
 \[
 \{N\leq X \; | \; \#([N, N+f(N)]\cap \mathbb{P})=m \}\gg \frac{X}{k^{2}\log k}e^{-3k^{4}\log k}\left( \frac{f(g^{-1}(X))}{\log g^{-1}(X)} \right)^{k}.
 \]
By the preceding arguments,  the above estimation holds with $k$ in (\ref{22}). Thus, the proof of Theorem 1.1 is completed.  \hspace{9.7cm}  $\Box$ \\

 \noindent
 ({\it Proof of Corollary} 1.2) In case that $f(x)=\lambda \log x$ for some $\lambda >0$ (independent of $x$), both (\ref{21.5}) and (\ref{28}) hold if $\lambda \ll k^{-4}\log ^{-2}k$. In this case,
 \[
 \frac{f(g^{-1}(x))}{\log g^{-1}(x)}=\frac{\lambda \log g^{-1}(x)}{\log g^{-1}(x)}=\lambda ,
 \]
 hence for any sufficiently large $m$, we have
 \[
 \# \{n\leq x \; | \; \# ([n, n+\lambda \log n] \cap \mathbb{P})=m \}\gg \frac{x \lambda ^{k}}{k^{2}\log k}e^{-3k^{4}\log k}
 \]
 holds for any sufficiently large $x$, in terms of $m$,  with $k$ in (\ref{22}).  \hspace{10.3cm}  $\Box$ \\

 \noindent
 ({\it Proof of Corollary} 1.3) Since $f$ is non-decreasing and $f(x)\geq 1$, the function $g(x)=20xf(x)$ satisfies $g^{-1}(x)\leq x$. The condition 
 \[
 \frac{f^{\prime}(x)}{f(x)}\leq \frac{1}{x\log x}
 \]
 implies that $\frac{f(x)}{\log x}$ is monotonically decreasing. Therefore, 
 \[
 \frac{f(g^{-1}(x))}{\log g^{-1}(x)}\geq \frac{f(x)}{\log x}.
 \]
 Hence the statement of Corollary 1.3 follows from Theorem 1.1. \hspace{4cm} $\Box$ \\

 \noindent
 ({\it Proof of Corollary} 1.4) If $f(x)=C_{k}$ with $C_{k}\geq 8e^{\gamma}k\log k$, then the  conditions of Theorem 1.1 hold and $g(x)=20C_{k}x$. Hence $g^{-1}(x)=\frac{x}{20C_{k}}$ and $f(g^{-1}(x))=C_{k}$. Thus the statement of Corollary 1.4 follows from Theorem 1.1. \hspace{7.2cm} $\Box$

%%%%%%%%%%%%%%%%%%%%%%%%%%%%%%%%%%%%%%%%%%%%%%%%%%%%%%%%%%%%%%%%%%%%%%%%%%%%%%%%%%%%%%%%%%%%%%%%%%%%%%%%%%%%
%%%%%%%%%%%%%%%%%%%%%%%%%%%%%%%%%%%%%%%%%%%%%%%%%%%%%%%%%%%%%%%%%%%%%%%%%%%%%%%%%%%%%%%%%%%%%%%%%%%%%%%%%%%%

\section{A lower bound for the ratio of integrals in Maynard's sieve}
In this section, we prove that the inequality (\ref{IJ}) holds for any sufficiently large $k$. First, we briefly see the explicit constructions of the functions $F$ and $F_{1}$ in Maynard's paper \cite{May2}, which are obtained by optimizing multidimensional Selberg sieve weight. Let $\psi :[0, \infty)\to [0, 1]$ be a smooth non-increasing function that has a compact support in $[0, 1]$ and equal to $1$ identically  in $[0, \frac{9}{10}]$. For $k\in \mathbb{Z}_{\geq 2}$, define $F: \mathbb{R}^{k}\to \mathbb{R}$ by
\[
F(t_{1}, \ldots ,t_{k})=\psi \left( \sum _{i=1}^{k}t_{i} \right)\prod _{i=1}^{k}\frac{\psi (\sqrt{k}t_{i})}{1+t_{i}k\log k },
\]
and define $F_{1}: \mathbb{R}^{k}\to \mathbb{R}$ by 
\[
F_{1}(t_{1}, \ldots ,t_{k})=\prod _{i=1}^{k}\frac{\psi (\sqrt{k}t_{i})}{1+t_{i}k\log k}.
\]
For a square-integrable function $G: \mathbb{R}^{k}\to \mathbb{R}$, put
\[
I_{k}(G):=\int _{0}^{\infty}\ldots \int _{0}^{\infty}G(t_{1}, \ldots ,t_{k})^{2} dt_{1}\cdots dt_{k},
\]
\[
J_{k}(G):=\int _{0}^{\infty}\ldots \int _{0}^{\infty}\left( \int _{0}^{\infty}G (t_{1}, \ldots ,t_{k}) dt_{k} \right)^{2}dt_{1}\cdots dt_{k-1}.
\]
Maynard \cite{May2} proved that the above $F$ and $F_{1}$ satisfy
\begin{equation}
\label{M1}
I_{k}(F)\leq I_{k}(F_{1})\ll I_{k}(F),
\end{equation}
\begin{equation}
\label{M2}
J_{k}(F)\leq J_{k}(F_{1})\ll J_{k}(F),
\end{equation}
and
\begin{equation}
\label{M3}
\frac{J_{k}(F_{1})}{I_{k}(F_{1})}=\frac{\log k}{4k}\left(1+O\left( \frac{1}{\log k} \right) \right).
\end{equation}
It follows from these estimates that there exists a positive constant $c_{IJ}=c_{IJ}(\psi)$ for which 
\begin{equation}
\label{M4}
\frac{J_{k}(F)}{I_{k}(F)}=c_{IJ}\frac{\log k}{k}\left(1+O\left( \frac{1}{\log k} \right) \right)
\end{equation}
holds. We see below that one can take $c_{IJ}\geq \frac{1}{4}-\varepsilon$ for any $\varepsilon >0$ with an appropriate choice of $\psi =\psi _{\varepsilon}$. 
\begin{prop}
For any $\varepsilon >0$, there exists an appropriate function $\psi$ for which (\ref{M4}) holds with $c_{IJ}\geq \frac{1}{4}-\varepsilon$. 
\end{prop}
\begin{proof}
Since there exists some function $\psi$ which satisfies the conditions above and is identically $1$ in $[0, 1-\delta]$ for any fixed $\delta >0$, to obtain the lower bound for $c_{IJ}$, it suffices to compute the integrals under the assumption that $\psi$ is the characteristic function of the interval $[0,1]$. Then, since $\psi (\sqrt{k}t_{i})$ is the characteristic  function of the interval  $[0, 1/\sqrt{k}]$, 
\begin{equation}
\label{M5}
\begin{aligned}
J_{k}(F_{1})&=\int _{0}^{\frac{1}{\sqrt{k}}}\ldots \int _{0}^{\frac{1}{\sqrt{k}}} \left( \int _{0}^{\frac{1}{\sqrt{k}}} \frac{dt_{k}}{1+t_{k}k\log k} \right)^{2}\frac{dt_{1}\cdots dt_{k-1}}{\prod _{i=1}^{k-1}(1+t_{i}k\log k)^{2}} \\
&=\left( \frac{\log (1+\sqrt{k}\log k)}{k\log k} \right)^{2}\left(\frac{1}{k\log k}\left(1-\frac{1}{1+\sqrt{k}\log k} \right) \right)^{k-1}.
\end{aligned}
\end{equation}
On the other hand, 
\begin{equation}
\label{M6}
\begin{aligned}
J_{k}(F)&\geq \underset{\substack{t_{1}+\ldots +t_{k}\leq 1-\frac{1}{\sqrt{k}} \\ 0\leq t_{1}, \ldots ,t_{k}\leq \frac{1}{\sqrt{k}}}}{\int \ldots \int} \left(\int _{0}^{\frac{1}{\sqrt{k}}}\frac{dt_{k}}{1+t_{k}k\log k} \right)^{2}\frac{dt_{1}\cdots dt_{k-1}}{\prod _{i=1}^{k-1}(1+t_{i}k\log k)^{2}} \\
&=\left( \frac{\log (1+\sqrt{k}\log k )}{k\log k} \right)^{2} \underset{\substack{t_{1}+\ldots +t_{k}\leq 1-\frac{1}{\sqrt{k}} \\ 0\leq t_{1}, \ldots ,t_{k}\leq \frac{1}{\sqrt{k}}}}{\int \ldots \int}   \frac{dt_{1}\cdots dt_{k-1}}{\prod _{i=1}^{k-1}(1+t_{i}k\log k)^{2}}.
\end{aligned}
\end{equation}
Put
\[
K_{k}(F):=\underset{\substack{t_{1}+\ldots +t_{k}\leq 1-\frac{1}{\sqrt{k}} \\ 0\leq t_{1}, \ldots ,t_{k}\leq \frac{1}{\sqrt{k}}}}{\int \ldots \int} \frac{dt_{1}\cdots dt_{k-1}}{\prod _{i=1}^{k-1}(1+t_{i}k\log k)^{2}}
\]
and decompose this by
\begin{equation}
\label{M7}
K_{k}(F)=K_{k}^{\prime}(F)-E_{k}(F),
\end{equation}
where
\begin{equation}
\label{M8}
K_{k}^{\prime}(F)=\underset{\substack 0\leq t_{1}, \ldots ,t_{k}\leq \frac{1}{\sqrt{k}} }{\int \ldots \int} \frac{dt_{1}\cdots dt_{k-1}}{\prod _{i=1}^{k-1}(1+t_{i}k\log k)^{2}}=\left(\frac{1}{k\log k} \left(1-\frac{1}{\sqrt{k}\log k} \right) \right)^{k-1},
\end{equation}
\begin{equation}
\label{M9}
E_{k}(F):=\underset{\substack{t_{1}+\ldots +t_{k}\geq 1-\frac{1}{\sqrt{k}} \\ 0\leq t_{1}, \ldots ,t_{k}\leq \frac{1}{\sqrt{k}}}}{\int \ldots \int} \frac{dt_{1}\cdots dt_{k-1}}{\prod _{i=1}^{k-1}(1+t_{i}k\log k)^{2}}.
\end{equation}
To evaluate $E_{k}(F)$, it is useful to keep
\[
\int _{0}^{\frac{1}{\sqrt{k}}}\frac{dt}{(1+tk\log k)^{2}}=\frac{1}{k\log k}\left(1-\frac{1}{1+\sqrt{k}\log k} \right)
\]
and 
\begin{align*}
\int _{0}^{\frac{1}{\sqrt{k}}}\frac{t}{(1+tk\log k)^{2}}dt  
&=\frac{1}{k\log k}\left(\int _{0}^{\frac{1}{\sqrt{k}}} \frac{dt}{1+tk\log k}-\int _{0}^{\frac{1}{\sqrt{k}}} \frac{dt}{(1+tk\log k)^{2}} \right) \\
&=\frac{\log (1+\sqrt{k}\log k)}{(k\log k)^{2}}-\frac{1}{(k\log k)^{2}}\left(1-\frac{1}{1+\sqrt{k}\log k} \right) \\
&\sim \frac{1}{2k^{2}\log k}
\end{align*}
in mind. In the integration of (\ref{M9}), since $t_{1}+\ldots +t_{k-1}\geq 1-\frac{1}{\sqrt{k}}$, for any $0\leq \mu <1-\frac{1}{\sqrt{k}}$, we have
\[
t_{1}+\ldots +t_{k-1}-\mu \geq 1-\frac{1}{\sqrt{k}}-\mu >0.
\]
Therefore, 
\[
\frac{1}{(1-\frac{1}{\sqrt{k}}-\mu )^{2}}(t_{1}+\ldots +t_{k-1}-\mu )^{2}\geq 1.
\]
We remove the condition $t_{1}+\ldots +t_{k-1}\geq 1-\frac{1}{\sqrt{k}}$ in (\ref{M9}) by multiplying this to the integrand. Thus 
\begin{equation}
\label{M10}
\begin{aligned}
E_{k}(F)&\leq \frac{1}{\left(1-\frac{1}{\sqrt{k}}-\mu \right)^{2}}\underset{\substack{0\leq t_{1}, \ldots ,t_{k-1}\leq \frac{1}{\sqrt{k}}}}{\int \ldots \int} \frac{(t_{1}+\ldots +t_{k-1}-\mu )^{2}}{\prod _{i=1}^{k-1}(1+t_{i}k\log k)^{2}}dt_{1}\cdots dt_{k-1} \\
&= \frac{1}{\left(1-\frac{1}{\sqrt{k}}-\mu \right)^{2}}\left\{ \sum _{j=1}^{k-1}I_{j}^{(1)}+\underset{j\neq l}{\sum _{j,l=1}^{k-1}}I_{j,l}^{(2)}-2\mu \sum _{j=1}^{k-1}I_{j}^{(3)}+\mu ^{2}I^{(4)} \right\},
\end{aligned}
\end{equation}
where
\[
I_{j}^{(1)}:=\underset{\substack0\leq {t_{1}, \ldots ,t_{k-1}\leq \frac{1}{\sqrt{k}}}}{\int \ldots \int} \frac{t_{j}^{2}}{\prod _{i=1}^{k-1}(1+t_{i}k\log k)^{2}}dt_{1}\cdots dt_{k-1},
\]
\[
I_{j,l}^{(2)}:=\underset{\substack{0\leq t_{1}, \ldots ,t_{k-1}\leq \frac{1}{\sqrt{k}}}}{\int \ldots \int} \frac{t_{j}t_{l}}{\prod _{i=1}^{k-1}(1+t_{i}k\log k)^{2}}dt_{1}\cdots dt_{k-1},
\]
\[
I_{j}^{(3)}:=\underset{\substack{0\leq t_{1}, \ldots ,t_{k-1}\leq \frac{1}{\sqrt{k}}}}{\int \ldots \int} \frac{t_{j}}{\prod _{i=1}^{k-1}(1+t_{i}k\log k)^{2}}dt_{1}\cdots dt_{k-1}
\]
for $j, l=1, \ldots ,k-1$, and
\[
I^{(4)}:=\underset{\substack{0\leq t_{1}, \ldots ,t_{k-1}\leq \frac{1}{\sqrt{k}}}}{\int \ldots \int} \frac{1}{\prod _{i=1}^{k-1}(1+t_{i}k\log k)^{2}}dt_{1}\cdots dt_{k-1}.
\]
By using $t_{j}^{2}\leq \frac{1}{\sqrt{k}}t_{j}$ for $t_{j}\in [0, 1/\sqrt{k}]$, it follows that 
\begin{align*}
I_{j}^{(1)}&\leq \frac{1}{\sqrt{k}}\int _{0}^{\frac{1}{\sqrt{k}}} \frac{t}{(1+tk\log k)^{2}} dt\left( \int _{0}^{\frac{1}{\sqrt{k}}}\frac{dt}{1+tk\log k} \right)^{k-2} \\
&=\frac{1}{\sqrt{k}}\left\{ \frac{\log (1+\sqrt{k}\log k)}{(k\log k)^{2}}-\frac{1}{(k\log k)^{2}}\left(1-\frac{1}{1+\sqrt{k}\log k} \right) \right\} \left\{\frac{1}{k\log k}\left(1-\frac{1}{1+\sqrt{k}\log k} \right) \right\}^{k-2} \\
&\leq (1+o(1))\frac{\log k}{2\sqrt{k}}\frac{1}{(k\log k)^{k}}, 
\end{align*}
and by direct computations, we have
\[
I_{j,l}^{(2)}=\frac{1}{4}(1+o(1))(\log k)^{2}\frac{1}{(k\log k)^{k+1}},
\]
\[
I_{j}^{(3)}=\frac{1}{2}(1+o(1))\log k \frac{1}{(k\log k)^{k}},
\]
and 
\[
I^{(4)}=(1+o(1))\frac{1}{k(\log k)^{k-1}}.
\]
Inserting these into (\ref{M10}), we have
\[
E_{k}(F)\leq \frac{1}{\left(1-\frac{1}{\sqrt{k}}-\mu \right)^{2}}\frac{(2\mu -1)^{2}}{4(k\log k)^{k-1}}+o\left( \frac{1}{(k\log k)^{k-1}} \right).
\]
By putting $\mu =\frac{1}{2}$, we obtain 
\begin{equation}
\label{M11}
E_{k}(F)=o\left(\frac{1}{(k\log k)^{k-1}} \right).
\end{equation}
Since by (\ref{M7}), (\ref{M8}) and (\ref{M11}) we have
\[
K_{k}(F)\sim \frac{1}{(k\log k)^{k-1}}.
\]
Therefore, by (\ref{M6}), 
\begin{equation}
\label{M12}
J_{k}(F)\geq (1+o(1))\frac{\log ^{2}(1+\sqrt{k}\log k)}{(k\log k)^{k+1}}.
\end{equation}
By (\ref{M5}) and (\ref{M12}), we have
\[
J_{k}(F)\geq (1+o(1))J_{k}(F_{1}),
\]
thus it follows that 
\[
\frac{J_{k}(F)}{I_{k}(F)}\geq (1+o(1))\frac{J_{k}(F_{1})}{I_{k}(F_{1})}=\frac{1}{4}(1+o(1))\frac{\log k}{k}.
\]
Actually, the function $\psi$ cannot be taken to be a characteristic function of the interval $[0,1]$, but  since the characteristic function can be approximated uniformly with infinite precision by $\psi$,  the above inequality is valid if we replace $\frac{1}{4}$ with  $\frac{1}{4}-\varepsilon$ for any $\varepsilon >0$. 
\end{proof}

%%%%%%%%%%%%%%%%%%%%%%%%%%%%%%%%%%%%%%%%%%%%%%%%%%%%%%%%%%%%%%%%%%%%%%%%%%%%%%%%%%%%%%%%%%%%%%%%%%%%%%%%%%%%
%%%%%%%%%%%%%%%%%%%%%%%%%%%%%%%%%%%%%%%%%%%%%%%%%%%%%%%%%%%%%%%%%%%%%%%%%%%%%%%%%%%%%%%%%%%%%%%%%%%%%%%%%%%%

\section{Acknowledgements}
This work is partially supported by the JSPS, KAKENHI Grant Number 24K06697. The author sincerely thanks  the referee for reading this paper carefully and for providing valuable comments and suggestions. 

%%%%%%%%%%%%%%%%%%%%%%%%%%%%%%%%%%%%%%%%%%%%%%%%%%%%%%%%%%%%%%%%%%%%%%%%%%%%%%%%%%%%%%%%%%%%%%%%%%%%%%%%%%%%
%%%%%%%%%%%%%%%%%%%%%%%%%%%%%%%%%%%%%%%%%%%%%%%%%%%%%%%%%%%%%%%%%%%%%%%%%%%%%%%%%%%%%%%%%%%%%%%%%%%%%%%%%%%%

\noindent
Kanto Gakuin University, \\
Kanazawa, Yokohama,\\
Kanagawa, Japan\\
E-mail address: sono@kanto-gakuin.ac.jp

\end{document}